\documentclass[twoside]{article}

\usepackage{amsmath,amsthm,amssymb}

\usepackage{mathptmx}

\usepackage[text={12.5cm,19cm},centering,paperwidth=17cm,paperheight=24cm]{geometry}

\usepackage[fontsize=10.4pt]{scrextend}

\def\titlerunning#1{\gdef\titrun{#1}}

\makeatletter
\def\author#1{\gdef\autrun{\def\and{\unskip, }#1}\gdef\@author{#1}}
\def\address#1{{\def\and{\\\hspace*{15.6pt}}\renewcommand{\thefootnote}{}\footnote{#1}}\markboth{\autrun}{\titrun}}
\makeatother

\def\email#1{email: \href{mailto:#1}{#1} }

\newenvironment{dedication}{\itshape\center}{\par\medskip}

\newtheorem{thm}{Theorem}[section]

\newtheorem{prop}[thm]{Proposition}
\theoremstyle{definition}

\newtheorem{exa}[thm]{Example}

\numberwithin{equation}{section}

\usepackage[hyperfootnotes=false,colorlinks=true,allcolors=blue]{hyperref}

\newcommand{\beq}{\begin{eqnarray}}
\newcommand{\eeq}{\end{eqnarray}}
\newcommand{\bq}{\begin{equation}}
\newcommand{\eq}{\end{equation}}
\newcommand{\beqn}{\begin{eqnarray*}}
\newcommand{\eeqn}{\end{eqnarray*}}

\def\e{\mathop{\rm e}\nolimits}
\def\Id{\mathop{\rm Id}\nolimits}

\def\DD{\mathop{\bf D\kern 0pt}\nolimits}

\def\SS{\mathop{\bf S\kern 0pt}\nolimits}
\def\ZZ{\mathop{\bf Z\kern 0pt}\nolimits}
\def\TT{\mathop{\bf T\kern 0pt}\nolimits}
\def\virgp{\raise 2pt\hbox{,}}
\def\cdotpv{\raise 1pt\hbox{ ;}}

\def\eps{\varepsilon}
\def\beq{\begin{equation}}
\def\eeq{\end{equation}}

\def\cdotv{\raise 2pt\hbox{,}}

\def\H{\mathop{\mathbb H\kern 0pt}\nolimits}
\def\R{\mathop{\mathbb R\kern 0pt}\nolimits}

\def\C{{\mathbb C}}% complex numbers
\def\N{{\mathbf N}}% nonnegative integers
\def\Sch{{\mathcal S}}% Schwartz space
\def\U{\mathcal U}

\def\virgp{\raise 2pt\hbox{,}}

\def\({\left(}
\def\){\right)}
\def\<{\left\langle}
\def\>{\right\rangle}
\def\le{\leqslant}
\def\ge{\geqslant}
\newcommand{\1}{\mathbb  I}

\begin{document}

% Give an abbreviation of the title for the running page headers.
\titlerunning{Adiabatic and non adiabatic evolution...}

% Here you can enter the full article title.
\title{\textbf{Adiabatic and non-adiabatic evolution of wave packets and applications to initial value representations}}

% Here you can enter the full names of authors separated by \and.
\author{Clotilde Fermanian Kammerer  \and Caroline Lasser   \and Didier Robert}

% Please do not enter a date.
\date{}

\maketitle

% Here you can enter the address and email of each author separated by \and by following the example.
\address{C. Fermanian Kammerer: Universit\'e Paris Est Cr\'eteil; \email{clotilde.fermanian@u-pec.fr} \and C. Lasser: Technische Universit\"at M\"unchen; \email{classer@ma.tum.de} \and
 D. Robert: Universit\'e de Nantes; \email{didier.robert@univ-nantes.fr}}

% Here you can enter an optional dedication.
\begin{dedication}
To Ari Laptev for his $70^{\rm th}$ birthday
\end{dedication}

%%%%%%%%%%%%%%
% Here you can enter the abstract, MSC classes, and keywords.
%\begin{abstract}
%An abstract.  
%\subjclass{Primary 12X34; Secondary 56Y78}
%\keywords{Aaaa, bbbb, cccc}
%\end{abstract}

\begin{abstract} We review some recent results obtained for  the  time evolution of wave packets 
%\deleted{for propagators associated with a} 
for systems of equations of pseudo-differential type, including Schr\"odinger ones, and discuss their application to the approximation of the associated 
%\deleted{semi-group} 
unitary propagator.  We start with scalar equations,  propagation of coherent states, and applications to the Herman--Kluk approximation. 
%\footnote{I have put ``thawed and frozen '' instead of ``Herman-Kluk'' but it could not be a good idea, just let me know...}, 
Then we discuss the extension of  these results to systems with eigenvalues of constant multiplicity or with  smooth crossings.
 \end{abstract}

\maketitle

%\tableofcontents

\section{Introduction}\label{intro}
We consider semi-classical systems of the form 
\begin{equation}\label{system}
i\eps \partial_t \psi^\eps(t) = \widehat {H(t)} \psi^\eps(t),\;\;\psi^\eps_{|t=t_0}=\psi^\eps_0
\end{equation}
where $(\psi^\eps_0)$ is a bounded family in $L^2(\R^d,\C^N)$, $\| \psi^\eps_0\|_{L^2(\R^d,\C^N)} =1$ and  $\hat H(t)$ is the $\eps$-Weyl quantization  of a smooth Hermitian symbol $H(t,x,\xi)$ satisfying suitable growth assumptions. We are interested in approximate realizations of the unitary propagator associated with this equation relying on the use of continuous superpositions of Gaussian wave packets. Before explaining these ideas in more detail (in particular the definition of the quantization that the reader will find below), let us first  emphasize different types of systems  that we have in mind.

\medskip

The simplest 
 case is the one where $H(t)$ is independent of $(x,\xi)$, which implies that  the operator $\hat H(t)$ coincides with the matrix $H(t)$.
In this case it is known that for $N\geq 2$ new phenomena appear by comparison with the scalar case $N=1$. 
When the eigenvalues of $H(t)$ are not crossing, the adiabatic  theorem says that
 if $\psi^\eps_0$  is an eigenfunction of $H(t_0)$, then at every time $\psi^\eps(t)$ has an asymptotic expansion  for $\epsilon\rightarrow 0$, and the leading term is an eigenfunction of $H(t)$. The first complete proof was given  by Friedrichs (\cite{Fri56}  after first  results by Born-Fock \cite{BoFo28} and Kato \cite{Ka50}. In \cite[Part II]{Fri56}, Friedrichs considered
  a non-adiabatic   situation where $H(t)$ is a $2\times 2$ Hermitian matrix with two analytic  eigenvalues $h_\pm(t)$ such that 
  $h_+(t)\neq h_-(t)$ for $t\neq 0$ and
  $$(h_+-h_-)(0)=0,\;\; \frac{d}{dt}(h_+-h_-)(0)\neq 0 .$$
  
In this spirit, we here consider a toy model for  space dependent Hamiltonians of the following form: We choose $d=1$, $N=2$, $\theta\in{\R}_+$, $k\in\R^*$, and
\beq\label{ToMo}
\widehat H_{k,\theta}=\frac{\eps}{i}\frac{d}{dx}\1_{\C^2} + kx\begin{pmatrix} 0&{\rm e}^{i\theta x}\\ {\rm e}^{-i\theta x}&0\end{pmatrix}.
\eeq
Its semiclassical symbol $H_{k,\theta}(x,\xi)=\xi +kx \begin{pmatrix} 0&{\rm e}^{i\theta x}\\ {\rm e}^{-i\theta x}&0\end{pmatrix}$
 has the eigenvalues and associated eigenvectors (the latter depending only in $x$)
 \begin{equation}\label{eigenvalues:TM}
 h_\pm (x,\xi) = \xi \pm kx,\qquad 
 \vec V_\pm(x) = \frac{1}{\sqrt 2}\begin{pmatrix}{\rm e}^{i\theta x} \\ \pm 1\end{pmatrix}
 \end{equation}
 So for $k\neq 0$ we have a  crossing at $x=0$ that we call {\it smooth crossing} because there exists smooth eigenvalues and eigenprojectors. We shall use this toy-model all along this article.
 Notice that one can prove that the operator $\widehat H_{k,\theta}$ is essentially self-adjoint in $L^2(\R, \C^2)$  by using a commutator criterion with the standard harmonic oscillator (see \cite[Appendix~A]{MaRo}).
As we shall see later,
the solutions of the  Schr\"odinger  equation for $\widehat H_{k,\theta}$ can be   computed
 by solving an ODE  asymptotically  as $\eps\rightarrow 0$, see \eqref{eq:toto}, and by  using non-adiabatic  asymptotic  results from Friedrichs \cite{Fri56} and Hagedorn \cite{Hag89}. 
 
 \vskip 0.2cm

More generally, of major interest because of their applications to molecular dynamics, are
  the Schr\"odinger Hamiltonians with matrix-valued potential considered by Hagedorn in his monograph \cite[Chapter~5]{Hag94},
\begin{equation}\label{ex:hag}
\widehat H_S = -\frac{\eps^2}{2}\Delta_x\, \1_{\C^N} + V(x),\quad V\in
{\mathcal C}^\infty(\R^d,\C^{N\times N}),
\end{equation}
or the models  arising in solid state 
physics in the context of Bloch band decompositions studied by  Watson and Weinstein in~\cite{WW}.
\begin{equation}\label{ex:WW}
\widehat H _A= A(-i\eps\nabla_x) + W (x) \1_{\C^2},\quad A\in{\mathcal C}^\infty(\R^d,\C^{N\times N}),\quad 
W\in{\mathcal C}^\infty(\R^d,\C). 
\end{equation} 

\medskip

The notation  $\widehat H$ refers to the Weyl quantization that we shall use extensively in this article. 
For $a\in{\mathcal C}^\infty(\R^{2d}, \C^{N,N})$ with adequate control on the growth of derivatives, the operator $\widehat a$ is defined 
 by its action on functions $f\in{\mathcal S}(\R^d,\C^N)$:
$${\rm op}^w_\eps(a)f(x):= \widehat a f(x) := (2\pi\eps)^{-d} \int_{\R^d} a\left({x+y\over 2}, \xi\right) {\rm e}^{{i\over \eps} \xi\cdot(x-y)} f(y) dy\, d\xi.$$ 
It turns out that the propagator ${\mathcal U}_H^\eps(t,t_0)$ associated with $\widehat H$  is well defined according to~\cite[Theorem~5.15]{MaRo} provided 
  that the map $(t,z)\mapsto H(t,z)$ is in ${\mathcal
  C}^\infty(\R\times \R^{2d},\C^{N\times N})$ valued in the set of self-adjoint matrices  and that it has  subquadratic growth, i.e.
\begin{equation}\label{hyp:H}
\forall \alpha\in \N^{2d} ,\;\;|\alpha|\geq 2,\;\;
\exists C_\alpha>0,\;\;\sup_{(t,z)\in\R\times \R^{2d}}\| \partial^\alpha H(t,z) \|_{\C^{N,N}}\leq C_\alpha.
\end{equation}
Using the commutator methods of~\cite{MaRo}, we could extend main of our results to a more general setting.
However, the assumptions~\eqref{hyp:H} are enough to  guarantee the existence of solutions to equation~(\ref{system}) in $L^2(\R^d,\C^N)$. 
One of our objectives here is to describe different Gaussian-based approximations of the semi-group  ${\mathcal U}_H^\eps(t,t_0)$ in the limit $\eps\rightarrow 0$.

\medskip

Let  $g^\eps_z$ denote the Gaussian wave packet centered in $z=(q,p)\in\R^{2d}$ with standard deviation $\sqrt \eps$:
\begin{equation}\label{def:gz}
g^\eps_z(x)= (\pi\eps)^{-d/4} {\rm exp}\left( -{1 \over 2 \eps}|x-q|^2 +{i\over \eps}p\cdot (x-q) \right).
\end{equation}
The family of wave packets $(g^\eps_z)_{z\in\R^{2d}}$ forms a continuous frame and 
provides for all square integrable functions $f\in L^2(\R^d)$ the reconstruction formula
\begin{equation}\label{reconstruction}
f(x)= (2\pi\eps)^{-d} \int_{z\in\R^{2d}}\langle g^\eps_z,f\rangle g^\eps_z (x) dz.
\end{equation}
The leading idea is then to write the unitary propagation of general, square integrable initial data $\psi^\eps_0\in L^2(\R^d)$ as 
\begin{equation}\label{unitaryprop}
{\mathcal U}^\eps_H(t,t_0)\psi^\eps_0=(2\pi\eps)^{-d}  \int_{z\in\R^{2d}} \langle  g^\eps_z,\psi^\eps_0\rangle 
\,{\mathcal U}^\eps_H(t,t_0)g^\eps_z\, dz,
\end{equation}
and to take advantage of the specific properties of the propagation of Gaussian states to obtain an integral representation that allows in particular for an efficient numerical realization of the propagator.

\medskip 
Such a program has been completely accomplished in the scalar case ($N=1$). It appeared first in the 80's in  theoretical chemistry~\cite{HK,Kay1,Kay2} and has led to the so-called Herman--Kluk approximation.   The mathematical proof of the convergence of this approximation is more recent~\cite{RS,R}.
Here, we revisit these results in Section~\ref{sec:scalar} and Section~\ref{sec:proofs}, and discuss some extensions 
to the case of systems ($N\ge1$), first for the gapped situation in Section~\ref{sec:adia} and then for smooth crossings in Section~\ref{sec:WPC}.

%%%%%%%%%%%%%%%%%%%%%%%%%%%%%%%%%%%%%%%%%%%%%

\section{Propagation of  Gaussian states and the Herman--Kluk approximation for scalar equations}\label{sec:scalar}

In this section,  $N=1$ and the equation~\eqref{system} is associated with a scalar Hamiltonian $H(t)= h(t)$ of subsquadratic growth \eqref{hyp:H}. In that case, the approximate propagation of Gaussian states is described by classical quantities, leading to a simple Herman--Kluk approximation.   We set
\[
J = \begin{pmatrix}0 & {\1}_{\R^d}\\ -{\1}_{\R^d} & 0\end{pmatrix}
\]
and for $z_0\in\R^{2d}$ we consider the classical Hamiltonian 
trajectory $z(t) = (q(t),p(t))$ defined by the ordinary differential equation
\[
\dot z(t) = J \partial_z h(t,z(t)),\;\; z(t_0)=z_0.
\]
 The associated flow map is then denoted by
\begin{equation}
\label{def:flot1}
\Phi_h^{t,t_0}(z_0) = z(t)=(q(t),p(t)),
\end{equation}
and the blocks of its Jacobian matrix 
$F(t,t_0,z_0) = \partial_{z}\Phi_h^{t,t_0}(z_0)$ 
by
\begin{align*}
F(t,t_0,z_0) = \begin{pmatrix}  A(t,t_0,z_0) &B(t,t_0,z_0) \\ C(t,t_0,z_0) &D(t,t_0,z_0)\end{pmatrix}.
\end{align*}
We note that the Jacobian satisfies the linearized flow equation
\begin{align}\label{eq:F}
\partial_t F(t,t_0,z_0) &= J {\rm Hess}_zh(t,z(t)) \, F(t,t_0,z_0),\;\;
F(t_0,t_0,z_0) = \1_{\R^{2d}}.
\end{align}
Thus, $F$ is a smooth in $t,t_0,z$ with any derivative in $z$ bounded. 
We will also use the action integral 
\begin{equation}
\label{def:S} 
\partial_t S(t,t_0,z_0) = p(t)\cdot \dot q(t)-h(t,z(t)),\;\; S(t_0,t_0,z_0) = 0.
\end{equation} 
Then, the Herman--Kluk approximation (also called {\it frozen Gaussian approximation} in the literature) of the unitary propagator ${\mathcal U}^\eps_h(t,t_0)$ writes as 
follows.

\begin{thm}[\cite{RS,R}]\label{thm:scalarhk}
Assuming \eqref{hyp:H}, the evolution through the scalar equation with  $N=1$ and $H=h\1_{\C} $ in~(\ref{system}) satisfies for every $J\geq 0$,
\[
{\mathcal U}^\eps_h(t,t_0) = {\mathcal I}^{\eps,J}_h(t,t_0) +  O(\eps^{J+1})
\]
in the norm of bounded operators on $L^2(\R^d)$, where the Herman--Kluk propagator is defined for 
all $\psi\in L^2(\R ^d)$,  
\beq\label{HK1}
{\mathcal I}_h^{\eps,J}(t,t_0)\psi = (2\pi\eps)^{-d} \int_{\R^{2d}} 
\langle g^\eps_z,\psi\rangle  u^{\eps,J}(t,t_0,z) {\rm e}^{\frac{i}{\eps}S(t,t_0,z)} g^\eps_{\Phi^{t,t_0}_h(z)} dz ,
\eeq
with 
$$
 u^{\eps,J}(t,t_0,z)=\sum_{0\leq j\leq J}\eps^j u_j(t,t_0,z),
$$
where every $u_j$ is smooth in $t, t_0,z$, with any derivative in $z$  bounded. The function $u_0$ is the Herman--Kluk prefactor,
\begin{equation}\label{def:HKprefactor}
u_0(t,t_0,z)= 2^{-d/2}   \ {\rm det}^{1/2} \left( A(t,t_0,z)+D(t,t_0,z)+i(C(t,t_0,z)-B(t,t_0,z))     \right),
\end{equation}
which has the branch of the square root determined by continuity in time. 
\end{thm} 

\medskip

Let us remark that the Gaussian wave packets  in \eqref{HK1} all have the  same covariance matrix $\Gamma = \1_{C^d}$, that is, in the terminology put forward by E. Heller \cite{H}, the Gaussians are frozen.  
The first statement of the form \eqref{HK1} at leading order ($J=0$) is due to M.~Herman and E. Kluk (1984)\cite{HK}. Then, 
W.H. Miller (2002) \cite{M} noticed that \eqref{HK1} can be deduced from the Van Vleck approximation of the propagator (hence related with a Feynman integral). In more recent work \cite{RS, R}, \eqref{HK1} was established  
using Fourier-Integral operator techniques.
Here we shall comment in more details on a proof  based on the propagation of  Gaussian wave packets, via a thawed Gaussian approximation (see~Section~\ref{sec:proofs}). This proof has the advantage that it can be  used in the case of systems of Schr\"odinger equations, which were not treated in the preceding references.

\medskip 

 A Gaussian wave packet is a wave packet with profile function belonging to the class of Gaussian states with variance taken in the Siegel set ${\mathfrak S}^ +(d)$ of  $d\times d$ complex-valued symmetric matrices with positive imaginary part,
\[
{\mathfrak S}^+(d) = \left\{\Gamma\in\C^{d\times d},\  \Gamma=\Gamma^\tau,\ \Im\Gamma >0\right\}.
\]
With  $\Gamma\in{\mathfrak S}^+(d)$ and $z\in\R$, we associate the Gaussian state
 \[
g_z^{\Gamma,\eps}(x) = c_\Gamma\, (\pi\eps)^{-d/4}\, \exp\left(\frac{i}{\eps} p\cdot(x-q)+\frac{i}{2\eps}\Gamma (x-q)\cdot (x-q)\right),
\]
where 
$c_\Gamma={\rm det}^{1/4}(\Im\Gamma)$  
is a positive normalization constant in $L^2(\R^d)$. We note that the standardized Gaussian defined in 
\eqref{def:gz} satisfies $g_z^\eps = g_z^{i\Id,\eps}$.
%\textcolor{cyan}{We could erase the following lines (Didier) For each matrix $\Gamma\in{\mathfrak S}^+(d)$ there exist $d\times d$ complex-valued 
%invertible matrices~$Q$ and $P$ such that 
%\[
%\Gamma = PQ^{-1}\quad\text{and}\quad
%\begin{pmatrix}\Re\, Q & \Im\, Q\\ \Re\, P & \Im\, P\end{pmatrix}\
%\text{is symplectic.}
%\] 
%In particular, $QQ^*$ is real-valued, and 
%$\Im\Gamma = (QQ^*)^{-1}$ (see \cite[Section V.1]{Lub}). This alternative parametrization suggests to choose the normalization 
%constant of the corresponding Gaussian function as 
%$\widetilde c_{\Gamma} = {\rm det}^{-1/2}(Q)$. Since $Q$ is complex, 
%$\widetilde c_{\Gamma}$ and $c_\Gamma$ might differ by a multiplicative factor 
%of modulus one.  }
%
%For each matrix $\Gamma\in{\mathfrak S}^+(d)$ there exist $d\times d$ complex-valued 
%invertible matrices~$Q$ and $P$ such that 
%\[
%\Gamma = PQ^{-1}\quad\text{and}\quad
%\begin{pmatrix}\Re\, Q & \Im\, Q\\ \Re\, P & \Im\, P\end{pmatrix}\
%\text{is symplectic.}
%\] 
%In particular, $QQ^*$ is real-valued, and 
%$\Im\Gamma = (QQ^*)^{-1}$ (see \cite[Section V.1]{Lub}). This alternative parametrization suggests to choose the normalization 
%constant of the corresponding Gaussian function as 
%$\widetilde c_{\Gamma} = {\rm det}^{-1/2}(Q)$. Since $Q$ is complex, 
%$\widetilde c_{\Gamma}$ and $c_\Gamma$ might differ by a multiplicative factor 
%of modulus one.  

\medskip

We shall use the notation ${\rm WP}^\eps_z(\varphi)$ to denote the bounded family in $L^2(\R^d)$ associated with $\varphi\in\mathcal S(\R^d)$ and $z=(q,p)\in\R^{2d}$ by 
 \begin{equation}\label{def:WP}
 {\rm WP}^\eps_z(\varphi)(x)= \eps^{-d/4}  {\rm e}^{\frac{i}{\eps} p\cdot(x-q)} \varphi \left( \frac {x-q}{\sqrt\eps} \right).
 \end{equation}
 With wave packet notation, the above Gaussian states satisfy $g_z^{\Gamma,\eps} = {\rm WP}^\eps_z(g_0^{\Gamma,1})$.

\begin{thm}[\cite{CR,Rcimpa}]\label{thm:propag1} Assuming \eqref{hyp:H} with $N=1$ and $H=h\1_\C$, there exists a family of time-dependent Schwartz functions $(\varphi_j(t))_{j\in\N}$ 
such that for all $N_0\in\N$, in $L^2(\R^d)$, 
\begin{equation}\label{eq:structure}
{\mathcal U}^\eps_h(t,t_0) g^{\Gamma_0,\eps}_{z_0}= {\rm e}^{\frac i\eps S(t,t_0,z_0)} \left(g^{\Gamma(t,t_0,z_0),\eps}_{\Phi^{t,t_0}(z_0)} +\sum_{j=1}^{N_0}\eps^{j/2} 
{\rm WP}^\eps_{\Phi^{t,t_0}(z_0)}(\varphi_j(t))\right) + O(\eps^{(N_0+1)/2} )
\end{equation}
 with 
\begin{equation}\label{def:Gamma}
\Gamma(t,t_0,z_0) 
= (C(t,t_0,z_0)+ D(t,t_0,z_0)\Gamma_0)(A(t,t_0,z_0) +B(t,t_0,z_0)\Gamma_0)^{-1}
\end{equation}
and $\displaystyle{ \;\;
c_{\Gamma(t,t_0,z_0)} 
= c_{\Gamma_0}\, {\rm det}^{-1/2}(A(t,t_0,z_0)+B(t,t_0,z_0)\Gamma_0),
}$ where the complex square root is continuous in time.
\end{thm}

Note that we have $\Gamma(t_0,t_0,z_0)=\Gamma_0$ in the statement above. 
%\begin{rem}
Actually, explicit information is obtained on the profiles $(\varphi_j(t))_{j\in\N}$, in particular about $\varphi_1(t)$ (see~\cite[Section~4.1.2]{CR} or Proposition~2.3 in~\cite{FLR1}). As we shall see in Section~\ref{sec:proofs}, this result can be 
a starting point for proving the Herman--Kluk approximation of Theorem~\ref{thm:scalarhk}. 
%\end{rem}

\medskip 

%\begin{rem}
The result above also holds for general wave packets as in~\eqref{def:WP} with profiles that are not necessarily Gaussians (see~\cite [Section~4.1.2]{CR}). Moreover, also the norm can be generalized to 
$$
\|f\|_{\Sigma^\eps_k} = \sup_{|\alpha|+|\beta| \leq k}\| x^\alpha (\eps \partial_x)^\beta f\|_{L^2},\quad k\in\N.
$$
%\end{rem}

\begin{exa} \label{clas:TM}
Both Theorem~\ref{thm:scalarhk} and \ref{thm:propag1} are exact  when the Hamiltonian~$h$ is quadratic. For the Hamiltonians $h_\pm(z) = p\pm kq$ associated with the toy model (see~\eqref{eigenvalues:TM}), the classical flow is linear 
$$\Phi_{h_\pm}^{t,t_0}(z_0) =:(q_\pm(t), p_\pm(t))=  z_0 + (t-t_0) (1,\mp k),\;\; z_0=(q_0,p_0),$$
the width of the Gaussian is constant, $\Gamma_\pm(t,t_0,z_0)=\Gamma_0$, and the actions only depend on $q_0$ and are given by 
$\displaystyle{S_\pm(t,t_0,q_0)= \mp k q_0(t-t_0)\mp \frac{k}2(t-t_0)^2.}$
\end{exa}

%%%%%%%%%%%%%%%%%%%%%%%%%%%%%%%%%%%%%%%%%%%%%%%%%%%%%%%%%%%%%%%%%%%%%%%%%

\section{Herman--Kluk formula in the adiabatic setting}\label{sec:adia}

We need adaptions for generalizing the scalar ideas to systems.
 The first one replaces the scalar Herman--Kluk prefactor $u_0(t,t_0,z_0)$ by a vector-valued one $\vec U(t,t_0,z_0)$ that is expanded in a basis of eigenvectors of $H(t,z)$, taken along the classical trajectory $\Phi^{t,t_0}_h(z)$ of the corresponding eigenvalue (denoted here by $h(t,z)$). This is done by parallel transport, and it is sufficient for an order $\eps$ approximation  as long as the system is gapped. 

\subsection{Parallel transport} 
 
 The following construction generalizes \cite[Proposition~1.9]{CF11}, which was inspired by the work of
G.~Hagedorn, see \cite[Proposition~3.1]{Hag94}. The details are given in \cite{FLR1}. In the sequel, we denote the complementary orthogonal projector 
by 
$\Pi^\perp(t,z) = \1_{\C^N}-\Pi(t,z)$ 
and assume that 
\begin{equation}\label{eq:decompose}
H(t,z) = h(t,z)\Pi(t,z) + h^\perp(t,z)\Pi^\perp(t,z)
\end{equation}
with the second eigenvalue given by 
$h^\perp(t,z) = {\rm tr}(H(t,z)) - h(t,z).$
We introduce the auxiliary matrices 
\begin{align}\nonumber
\Omega(t,z) &=-\tfrac12\big(h(t,z)-h^\perp(t,z)\big)\Pi(t,z)\{\Pi,\Pi\}(t,z)\Pi(t,z) ,\\*[1ex]
\nonumber
K(t,z) &= \Pi^\perp(t,z)\left(\partial_t\Pi(t,z)+\{h,\Pi\}(t,z)\right)\Pi(t,z),\\*[1ex]
\nonumber
\Theta(t,z) &= i\Omega(t,z) + i(K-K^*)(t,z),
\end{align}
that are smooth and satisfy some  algebraic properties. In particular, 
$\Omega$ is skew-symmetric and $\Theta$ is self-adjoint:
$\Omega = -\Omega^*$ and $\Theta = \Theta^*$.

\begin{prop}[\cite{FLR1}] \label{prop:eigenvector}
Let $H(t,z)$ be a smooth Hamiltonian that satisfies~\eqref{hyp:H} and has a smooth spectral decomposition~\eqref{eq:decompose}. We assume that both eigenvalues are of subquadratic growth as well. We consider $\vec V_0\in{\mathcal C}_0^\infty(\R^{2d},\C^N)$ and $z_0\in \R^{2d}$ such that there exits a neighborhood~$U$ of $z_0$ such that for all $z\in U$
$$\vec V_0 (z)=\Pi(t_0,z)\vec V_0(z)\quad\text{and}\quad \|\vec V_0(z)\|_{\C^N} = 1.$$
Then, 
  there exists a smooth normalized vector-valued function $\vec V(t,t_0)$ satisfying 
  $$\vec V(t,t_0,z)= \Pi(t,z)\vec V(t,t_0,z)\quad\text{for all}\quad z\in \Phi_h^{t,t_0}(U),$$
   such that for all $t\in\R$ and $z\in\Phi_h^{t,t_0}(U)$,
\begin{equation}\label{eq:eigenvector}
\partial_t \vec V(t,t_0,z) + \{h, \vec V\} (t,t_0,z) = -i\Theta(t,z)\vec V (t,t_0,z),\;\; \vec V(t_0,t_0,z) = \vec V_0(z).
\end{equation} 
\end{prop}

We note that the results of this  Proposition are valid as long as smooth eigenprojectors and eigenvalues do exist. It does not require an explicit  adiabatic situation and we will use  that observation in Section~\ref{sec:WPC}. 
In the case of Schr\"odinger systems, one refers to \eqref{eq:eigenvector} as parallel transport because the vectors 
$\partial_t \vec V(t) + \{h, \vec V\}(t)$ belong to the range of $\Pi^\perp(t)$ at any time of the evolution.

\begin{exa}\label{pt:TM} For the toy model $H_{k,\theta}(x,\xi)$, the auxiliary matrices are
\[
\Omega_\pm = 0,\quad K_\pm = \frac{i\theta}{4}\begin{pmatrix}1 & \pm {\rm e}^{i\theta x}\\ 
\mp {\rm e}^{-i\theta x} & -1\end{pmatrix},\quad \Theta_\pm = \frac\theta2\begin{pmatrix}-1 & 0\\ 0 & 1\end{pmatrix}.
\]
Initiating the parallel transport equation by the eigenvectors given in \eqref{eigenvalues:TM}, we obtain
\[
\vec V_\pm(t,t_0,q) = \frac{1}{\sqrt2}
\begin{pmatrix}{\rm e}^{i\theta/2(t-t_0)}\, {\rm e}^{i\theta(q-t+t_0)}\\ \pm{\rm e}^{-i\theta/2(t-t_0)}\end{pmatrix}
= \frac{1}{\sqrt2}
\begin{pmatrix}{\rm e}^{-\frac i 2\theta (t-t_0)}\, {\rm e}^{i\theta q}\\ \pm{\rm e}^{-\frac i2 \theta(t-t_0)}\end{pmatrix}.
\]
Note that they do not depend on~$p$.
\end{exa}

\subsection{Herman-Kluk approximation in an adiabatic setting }

In the context of the preceding section, and in presence of an eigenvalue gap,
\[
\exists\delta>0:\ {\rm dist}(h(t,z),h^\perp (t,z)\}) \geq \delta\ \text{for all}\ (t,z),
\]
it is well-known that one has {\it adiabatic decoupling}:
 If the initial data are 
scalar multiples of an eigenvector associated with the eigenvalue, then the solution keeps this property to leading 
order in $\eps$. For an approximation to higher order in $\eps$, one has to consider perturbations of the eigenprojector, the so-called {\it superadiabatic projectors} (see~\cite{bi,Te} and the new edition of~\cite{CR} that should appear soon). Adiabatic theory implies a  Herman--Kluk approximation of the propagator that we state next.
% \footnote{I have shortened here}

\begin{thm}\label{thm:decoupling}\cite{FLR1}
In the situation of~Proposition~\ref{prop:eigenvector},  
we assume the existence of an eigenvalue gap for the eigenvalue~$h$ of $H$ and 
consider initial data of the form 
\[
\psi^\eps_0= \widehat {\vec V_0} v^\eps_0 +O(\eps)\ \text{in}\  L^2(\R^d),
\] 
where $\vec V_0$ is a smooth eigenvector and 
$(v^\eps_0)_{\eps>0}$ a family of functions uniformly bounded in $L^2(\R^d,\C)$
Then, in $L^2(\R^d)$,
\[
\mathcal U^\eps_H(t,t_0)\psi^\eps_0 = (2\pi\eps)^{-d}  \int_{\R^{2d}}    
\langle g^\eps_z,v^\eps_0 \rangle\ \vec U(t,t_0,z) \ {\rm e}^{\frac{i}{\eps}S(t,t_0,z)} \ 
 g^\eps_{\Phi_h^{t,t_0}(z)} \ dz + O(\eps)
\]
$$\mbox{with} \;\qquad\vec U (t,t_0,z)=u_0(t,t_0,z) {\vec V(t,t_0,\Phi_h^{t,t_0}(z) )}$$
where $u_0(t,t_0,z)$ is given by~\eqref{def:HKprefactor} and the eigenvector $\vec V(t,t_0,z)$ by~\eqref{eq:eigenvector}.
%\footnote{I have shortened here}
%\Phi_h^{t,t_0}(z) )$ is determined by solving
%\begin{equation}\label{eq:eigenvec}
%i\partial_t \vec Y(t,z) = \Theta(t,t_0,\Phi_h^{t,t_0}(z)) \vec Y(t,z),\quad \vec Y(t_0,z) = \vec V_0(z).\footnote{notations to check, refer to (3.3)?}
%\end{equation}
\end{thm}

We describe in Section~\ref{sec:proofs} a proof of this result that crucially uses
%Theis result can be proved by a deformation argument as described below , using
 the approximate evolution of a Gaussian state, that is~\cite[Section~3]{Rcimpa},  
 \begin{align}\label{exp:adiabGW}
\mathcal U^\eps_H(t,t_0) (\widehat {\vec V_0}g^{\Gamma,\eps}_z)= &\ {\rm e}^{\frac i\eps S(t,t_0,z)} \widehat {\vec V(t,t_0)} g^{\Gamma(t,t_0,z),\eps}_{\Phi^{t,t_0}_h(z)} 
\left(1+\sqrt\eps \, \vec a(t) \cdot \frac{(x-q(t))}{\sqrt\eps}\right) +O(\eps)
%&+ \sum_{j=1}^{N_0}\eps^{j/2}{\rm e}^{\frac i\eps S_j(t)}  \widehat {\vec V_j(t)} {\rm WP}^\eps_{z_j(t)}(\varphi_j(t)) + O(\eps^{(N_0+1)/2} ),
\end{align}
%\begin{align}\nonumber 
%\mathcal U^\eps_H(t,t_0) (\widehat {\vec V_0}g^{\Gamma,\eps}_z)= &\ {\rm e}^{\frac i\eps S(t,t_0,z)} \widehat {\vec V(t,t_0)} g^{\Gamma(t,t_0,z),\eps}_{\Phi^{t,t_0}_h(z)}\\
%\label{exp:adiabGW}
%&+ \sum_{j=1}^{N_0}\eps^{j/2}{\rm e}^{\frac i\eps S_j(t)}  \widehat {\vec V_j(t)} {\rm WP}^\eps_{z_j(t)}(\varphi_j(t)) + O(\eps^{(N_0+1)/2} ),
%\end{align}
%\footnote{It should be $j\geq 2$, isn'it ?}
with $\vec a(t) = \vec a(t,t_0,z)$ a smooth and bounded map.  Note that using superadiabatic projectors~\cite{Rcimpa},  one can push the asymptotics further and exhibit $O(\eps)$ contributions that will have components on the other mode. 
%If the trajectory $\Phi^{t,t_0}_h(z)$ remains far from points with zero eigenvalue gap, then the formula 
%remains valid, but it fails as soon as the trajectory touches a crossing point. 
Here again, a formula similar to~\eqref{exp:adiabGW} can be proved for general wave packets~\cite{CR}.

\subsection{Algorithmic  realization of the propagator}\label{sec:alg_adiab}

Numerical realizations of the Herman--Kluk approximation have been first developed in~\cite{HKD} 
and are still in practical use in theoretical chemistry~\cite{WCG}. 
Below, we follow the more recent account given in~\cite{LS,LL}. We stay with  the assumptions of Theorem~\ref{thm:decoupling} and 
consider initial data  associated with the (gapped) mode $h$.
%, $\psi^\eps_0= \widehat {\vec V_0} v^\eps_0+O(\eps)$. The algorithm consists in three steps. 

\vskip 0.2cm

\noindent {\bf Step 1: Initial sampling}.
Choose a set of numerical quadrature points 
$$
z_1=(q_1,p_1),\ldots,z_N=(q_N,p_N)
$$ 
with associated weights $w_1,\ldots,w_N>0$, 
and evaluate the initial transform $\langle g_z^\eps, v^\eps_0\rangle$ in these points. This provides an approximation to the initial data, 
$$
\psi^\eps_0(x)\sim (2\pi\eps)^{-d} \sum_{1\leq j\leq N}  \langle g_{z_j}^\eps, v^\eps_0\rangle 
\vec V_0(z_j)  g_{z_j}^\eps(x) w_j,
$$
which is of order $\eps$ in $L^2(\R^d)$, as long as the chosen quadrature rule is sufficiently accurate.

\vskip 0.2cm

\noindent {\bf Step 2: Transport}.
For each of the points $z_j$, $j=1,\ldots,N$, compute 

\medskip
\noindent
(1) the classical trajectories $z_j(t)$ defined by \eqref{def:flot1},\\  
(2) the eigenvectors $\vec V(t,t_0,z_j(t))$ along the flow by use of equation~\eqref{eq:eigenvec},\\
(3) the Jacobian matrices $F(t,t_0,z_j)$ using equation~\eqref{eq:F},\\
(4) the action integrals $S(t,t_0,z_j)$ using equation~\eqref{def:S},\\
(5) the Herman-Kluk prefactor $u_0(t,t_0,z_j)$ using the Jacobians and equation~\eqref{def:HKprefactor}.\\*[1ex]
If the time-discretization is symplectic and sufficiently accurate, then the overall accuracy of order $\eps$ is 
not harmed, see \cite[Theorem~2]{LS}.

\vskip 0.2cm

\noindent {\bf Step 3: Conclusion}.
At the end of these thwo steps, we are left with the Hermann-Kluk quadrature formula:
$$\psi^\eps(t,x) \sim (2\pi\eps)^{-d} \sum_{1\leq j\leq N} \langle g_{z_j}^\eps, v^\eps_0\rangle \vec U(t,t_0,z) {\rm e}^{{i\over \eps} S(t,t_0,z_j)} g_{z_j(t)}^\eps w_j.$$

\medskip
Of course, the algorithm generalizes to initial data which has several components on separated eigenspaces. 
In higher dimensional applications, often Monte-Carlo quadrature is used. Then, the initial transform could be 
written as
$$(2\pi\eps)^{-d}\langle g_z^\eps,v^\eps_0\rangle dz = r^\eps_0(z) \mu^\eps_0(dz),$$
$$\mbox{where}\;\;\;\;\mu^\eps_0(dz)= \left( \int |\langle g_z^\eps, v^\eps_0\rangle | dz\right)^{-1}   |\langle g_z^\eps,v^\eps_0\rangle | dz$$
%where $\displaystyle{\mu^\eps_0(dz)= \left( \int |\langle g_z^\eps, v^\eps_0\rangle | dz\right)^{-1}   |\langle g_z^\eps,v^\eps_0\rangle | dz}$
is a probability measure and $r^\eps_0(z)$ a complex-valued function of $L^1(\R^{2d},d\mu^\eps_0)$ 
according to the Radon--Nikodym Theorem. The quadrature nodes $z_1,\ldots,z_N$ are then chosen independently and identically distributed according to the measure $\mu^\eps_0$, while the weights are all the same, 
$w_j= 1/N$ for all $j$.

%%%%%%%%%%%%%%%%%%%%%%%%%%%%%%%%%%%%%%%%%%%%%%%%%%%

\section{What about smooth crossings ?}\label{sec:WPC}

For  simplicity we assume here that the Hermitian matrix $H(t,z)$ is $2\times 2$ ($N=2$) (the same results are also proved
for two eigenvalues with arbitrary multiplicity \cite{FLR1}). We assume that it has smooth eigenvalues $h_1(t,z)$ and $h_2(t,z)$ and smooth eigenprojectors $\Pi_1(t,z)$ and $\Pi_2(t,z)$. To ensure control on the derivatives of the eigenprojectors, 
we suppose a non crossing assumption at infinity: 
there exist  $c_0,n_0,r_0>0$ such that
\begin{equation*}
%\label{hyp:gapinfinity}
|h_1(t,z)-h_2(t,z)| \ge c_0 \langle z\rangle^{-n_0}\ \text{for all}\ (t,z)\ \text{with}\ |z|\ge r_0,
\end{equation*}
where we denote $\langle z\rangle = (1+|z|^2)^{1/2}$.
We assume the following form of the matrix:

\medskip 

\noindent {\bf Asumption (SC)}.
There exist scalar functions
  $v,f\in{\mathcal C}^\infty(\R^{2d+1}, \R)$ and a vector-valued function $u\in{\mathcal C}^\infty(\R^{2d+1}, \R^3)$ with  
$|u(t,z)| = 1$ for all $(t,z)$ such that 
$$H(t,z)=v(t,z){\rm Id} + f(t,z) \begin{pmatrix} u_1(t,z) & u_2 (t,z) +iu_3(t,z) \\u_2(t,z) -iu_3(t,z) &-u_1 (t,z) \end{pmatrix}.$$
Besides, the crossing is {\it generic} in $\Upsilon$ in the sense that 
$$(\partial_tf + \{v,f\} )(t^\flat,z^\flat) \neq 0,\;\;\forall (t^\flat,z^\flat)\in \Upsilon.$$
%Under Assumption~(SC), the smooth eigenvalues of~$H$, 
%$h_1 = v+f$ and  $h_2 = v-f$,
%cross one another on the set 
%$\Upsilon = \{f=0\}$, 
%and we say that the smooth crossing on $\Upsilon$ is {\it generic} in $(t^\flat, z^\flat)\in\Upsilon$ if
%$$(\partial_tf + \{v,f\} )(t^\flat,z^\flat) \neq 0.$$
In that case, the crossing set $\Upsilon$ is a submanifold of codimension one, and all the classical trajectories that reach $\Upsilon$ are transverse to it. 

\begin{exa}
For the toy model~\eqref{ToMo}, $u(x) = (0,\cos(\theta x),\sin(\theta x))$, $v(\xi)=\xi$ and $f(\xi)=k\xi$.
%$$v(\xi)=\xi,\;\;f(x)=kx, \;\;  u(x) = (0,\cos(\theta x),\sin(\theta x)).$$
 Hence we have Assumption~(SC) : 
 $\Upsilon =\{x=0\}$ and $\{v,f\}=k\not=0$.
 %, the crossing is generic everywhere on $\Upsilon$ if $k\neq 0$. 
\end{exa} 
 
In contrast to the previous adiabatic situation, initial data associated with one eigenspace generates
a component on the other eigenspace, that is larger than the adiabatic $O(\eps)$, namely $O(\sqrt\eps)$. 

Starting at time $t_0$ with a Gaussian wave packet, that is associated with the eigenvalue $h_1$ and localized far from the crossing set~$\Upsilon$, an approximation of the form~\eqref{exp:adiabGW} holds as long as the trajectory does not reach~$\Upsilon$. 
The apparition of a $\sqrt\eps$ contribution on the other mode then occurs exactly on~$\Upsilon$. 
One can interpret this phenomenon in terms of {\it hops}: starting at time $t_0$ from some point $z_0\notin\Upsilon$ 
for which the Hamiltonian trajectory  $z_1(t,t_0)= \Phi^{t,t_0}_1(z_0)$ passes through the crossing at time $t=t^\flat(z_0)$ and point $z^\flat=z_1(t^\flat,t_0)$,  we generate a new trajectory  $\Phi_2^{t,t^\flat} (z^\flat)$ associated with the mode~$h_2$. This results in the construction of a  {\it hopping trajectory} that hops from one mode to the other one at time $t^\flat$.

Such an interpretation is crucial for describing the dynamics of systems with eigenvalue crossings. It  has been introduced around the 70s in the chemical literature for avoided and conical crossings of eigenvalues 
and has been widely used since then (see \cite{Tu12}). 
The first mathematical results on the subject are more recent and analyse the propagation of 
Wigner functions through singular crossings (see~\cite{FL08,FL17}).

 We now aim at a precise description of these new contributions in the case of smooth crossing, first for initial data that are Gaussian wave packets, then we lift this result to a Herman--Kluk formula for smooth crossings in the case of the toy model.

 \subsection{Wave packet propagation}

Let us now give a precise statement for the propagation of wave packets through a smooth crossing. 
We start with initial data of the form 
\begin{equation*}
%\label{initialdata}
\psi^\eps_0= \widehat{\vec V_0} v^\eps_0,\;\; v^\eps_0 = g^{\Gamma_0, \eps}_{z_0}
 \end{equation*}
 with 
 $H(t_0,z)\vec V_0(z)=h_1(t_0,z) \vec V_0(z)$ in a neighborhood of $z_0$. 
We use Proposition~\ref{prop:eigenvector} to construct two families of time-dependent eigenvectors: $(\vec V_1(t,z))_{t\geq t_0}$  is associated with the eigenvalue $h_1(t,z)$ and initial data at time $t_0$ given by $\vec V_1(t_0,z) = \vec V_0(z)$, while $(\vec V_2(t,z))_{t\geq t^\flat}$ is constructed  for $t\ge t^\flat$  (the crossing time introduced in the previous paragraph) for the eigenvalue $h_2(t,z)$ with initial data at time~$t^\flat=t^\flat(z_0)$ satisfying 
\begin{equation*}
%\label{def:V2}
 \vec V_2(t^\flat, z)= -\gamma (t^\flat,z)^{-1} {\Pi_2(\partial_t \Pi_2 +\{v,\Pi_2\} )\vec V_1}
 (t^\flat ,z)
\end{equation*}
\[{\rm with }\;\;
\gamma(t^\flat,z) =\|\left(\partial_t \Pi_2+\{v,\Pi_2\}\right)\vec V_1(t^\flat,z) \| _{\C^N} .
\]
We next introduce a family of  transformations, which describes the non-adiabatic effects for 
a wave packet that passes the crossing. For parameters $(\mu,\alpha,\beta)\in\R\times\R^{2d}$ and $\varphi\in\Sch(\R^d)$, we set 
 % \beq\label{transf1}
 \[
 {\mathcal T}_{\mu,\alpha, \beta}\varphi(y) =
  \int_{-\infty}^{+\infty}{\rm e}^{i(\mu-\alpha\cdot\beta/2)s^2} {\rm e}^{is\beta\cdot y}\varphi(y-s\alpha)ds .
 \]
 %\eeq
This operator maps $\mathcal S(\R^d)$ into itself if and only if $\mu\not=0$. Moreover, for $\mu\not=0$, it is a metaplectic transformation of the Hilbert space
 $L^2(\R^d)$, multiplied by   a complex number. 
In particular,  for any Gaussian function
  $g^\Gamma$, the function   
  $ {\mathcal T}_{\mu,\alpha, \beta}g^\Gamma $ is a Gaussian:
  \[ 
 {\mathcal T}_{\mu,\alpha, \beta}\,g^\Gamma = c_{\mu,\alpha,\beta,\Gamma}\, g^{\Gamma_{\mu, \alpha,\beta,\Gamma}},
 \]
   where $\Gamma_{\mu, \alpha,\beta,\Gamma}\in\mathfrak S^+(d)$ and $c_{\mu,\alpha,\beta,\Gamma}\in \C$ 
    can be computed explicitly (see \cite[Appendix E]{FLR1}).

\medskip 
Combining the parallel transport for the eigenvector and the metaplectic transformation for the non-adiabatic transitions, we obtain the following result. 

\begin{thm}[Propagation through a smooth crossing]\label{theo:WPcodim1}
Let Assumption~{\rm (SC)} on the Hamiltonian matrix $H(t)$ hold and that the crossing is generic. Assume that the initial data 
$(\psi^\eps_0)_{\eps>0}$ are wave packets\
 as above.
  Let $T>0$ be such that the interval $[t_0,t^\flat] $ is strictly included in the interval $[t_0,t_0+T]$. Then, for all $k\in\N$ there exists a constant $C>0$ such that for all $t\in  [t_0,t^\flat)\cup(t^\flat,t_0+T]$ and for all $\eps\leq |t-t^\flat|^{9/2}$,
$$\left\| \psi^\eps(t) - \widehat {\vec V}_1(t) v^\eps_1(t) -\sqrt\eps {\bf 1}_{t>t^\flat} \widehat {\vec V}_2(t) v^\eps_2(t) \right\|_{L^2(\R^d)}\leq C \,\eps^{m},$$
with an exponent $m\ge 5/9$. The components of the approximate solution are 
\[
v^\eps_1(t) = \U_{h_1}^\eps(t,t_0) g^{\Gamma_0,\eps}_{z_0}\quad \text{and}\quad v^\eps_2(t) =\U_{h_2}^\eps(t,t^\flat)v^\eps_2(t^\flat)
\]
%\begin{equation}\label{eq:v2}
\[
\mbox{with}\qquad v^\eps_2(t^\flat)= \gamma^\flat {\rm e}^{iS^\flat/\eps}{\rm WP}^\eps_{z^\flat}({\mathcal T}^\flat
\varphi_0(t^\flat)),
\]
%\end{equation}
where $\varphi_0(t) = g^{\Gamma_1(t,t_0,z),1}_{0}$ is the leading order profile of the coherent state 
$v^\eps_1(t)$ given by Theorem~\ref{thm:propag1}, and 
%\begin{equation} \label{def:alpha} 
\[
\gamma^\flat = \gamma(t^\flat,z^\flat) = \|\left(\{v,\Pi_2\}+\partial_t\Pi_2\right)\vec V_1(t^\flat,z^\flat) \| _{\C^N} .
\]
%\end{equation}
The transition operator
$\mathcal T^\flat= \mathcal T_{ \mu^\flat, \alpha^\flat,\beta^\flat}$
is defined by the parameters
%\begin{equation}\label{def:lambda}
\[
\mu^\flat = \tfrac{1}{2}\left( 
\partial_t f+\{v,f\}\right)(t^\flat, z^\flat)\;\; \text{and}\;\;
(\alpha^\flat,\beta^\flat) = Jd_zf(t^\flat, z^\flat). 
\]
%\end{equation}
The constant $C = C(T,k,z_0,\Gamma_0)>0$ is $\eps$-independent but depends on the 
Hamiltonian $H(t,z)$, the final time $T$, and on the initial wave packet's center~$z_0$ and width~$\Gamma_0$.
\end{thm}

Note that by the transversality assumption we have  $\mu^\flat\not=0$,  
which guarantees that ${\mathcal T}^\flat
\varphi_0(t^\flat)$ is Schwartz class. 
The coefficient $\gamma^\flat$ quantitatively describes the distortion of the eigenprojector~$\Pi_1(t)$ during its evolution along the flow generated by $h_1(t)$.
If the matrix $H(t,z)$ is diagonal (or diagonalizes in a fixed orthonormal basis that is $(t,z)$-independent), then  $\gamma^\flat=0$: the equations are decoupled (or can be decoupled), and one can then apply the result for a system of two independent equations with a scalar Hamiltonian and, of course, there is no interaction of order $\sqrt\eps$ between the modes.

The previous theorem extends to more general wave packets as defined in~\eqref{def:WP} and also holds with respect to $\Sigma_k^\eps$-norms for $k\in\N$ (see~\cite[Theorem~3.8]{FLR1}). As mentioned alongside the proof \cite{FLR1}, the argument contains the germs for a full asymptotic expansion in powers of $\sqrt\eps$ (with $\log\eps$ corrections).

\begin{exa}\label{gamma:TM}
Notice that for the toy model $H_{k,\theta}$, 
 we have at any point of $\Upsilon= \{x=0\}$,
$$\mu^\flat = \frac k2,\;\;\alpha^\flat =0,\;\;\beta^\flat =-k,\;\;\gamma^\flat = \frac {|\theta|}{2}, $$
and 
$\displaystyle{\mathcal T^\flat\varphi (y)= \sqrt { \frac {2\pi}{ik}} {\rm e} ^{ \frac k {2i} y^2} \varphi(y)}$ for all 
$\varphi\in \mathcal S(\R)$ and all $y\in\R$. Besides, if $t_0\leq 0$, the trajectories of the minus mode that reach~$\Upsilon$ are those arising from points $z=(q,p)$ with $q<0$. One then has $t^\flat=t_0-q$, $p^\flat= p- kq$ and  the trajectory on the plus mode issued from $z^\flat=(0,p^\flat)$ is 
\begin{equation*}
\Phi^{t,t^\flat}_+ (0,p^\flat)= (t-t^\flat, p^\flat - k(t-t^\flat)= (t-t_0+q, p-2kq-k(t-t_0)).
\end{equation*}
\end{exa}

\subsection{Towards a Herman--Kluk approximation}

The preceding result implies that the leading order of the propagation is still driven by the modes in which the initial data had been taken and we have an Herman--Kluk formula similar to the one obtained in the adiabatic regime, however, with a remainder which is worse. One can conjecture that a more accurate Herman--Kluk approximation  holds in a weaker sense (see~\cite{FLR2}). 
We define the operator $\mathcal I^\eps_{sc}(t)$ by its actions on functions of the form
\[
\psi^\eps_0= \widehat {\vec V_0} v^\eps_0 +O(\sqrt \eps)\ \text{in}\  L^2(\R^d)
\] 
with $v^\eps_0\in L^2(\R^d)$ as
\begin{align*}
&\mathcal I^\eps_{sc}(t) \psi^\eps_0(x)= (2\pi\eps)^{-d}  \int_{\R^{2d}}    
\langle g^\eps_z,v^\eps_0 \rangle\ \vec U_1(t,t_0,z) \ {\rm e}^{\frac{i}{\eps}S_1(t,t_0,z)} \ 
 g^\eps_{\Phi_{h_1}^{t,t_0}(z)} \ dz  + 
\sqrt\eps \,  (2\pi\eps)^{-d} \\
&\;\qquad\times \,\int_{\R^{2d}}    {\bf 1} _{ t>t^\flat(z)} 
\langle g^\eps_z,v^\eps_0 \rangle\ \vec U_2(t,t^\flat(z),z)   {\rm e}^{\frac i\eps S_1(t^\flat(z), t_0, z) + \frac{i}{\eps}S_2(t,t^\flat(z),\mathfrak z^\flat(z))} \ 
 g^\eps_{\Phi_{h_2}^{t,t^\flat(z)}(\mathfrak z^\flat)} dz
 \end{align*}
 with $\mathfrak z^\flat(z)= \Phi_{h_1}^{t^\flat(z), t_0}(z)$ and with some adequate formula (taking into account   the transfer coefficients  $\gamma^\flat(z)$)  for the prefactor $ \vec U_2(t,t^\flat(z),z)= v_2(t,t^\flat(z),z) \vec V_2(t, t^\flat(z), \mathfrak z^\flat(z) )$. 
 
 \medskip

The conjecture is, that if Assumptions~(SC) are satisfied, then, for all $\chi\in \mathcal C_0^\infty(\R)$, in $L^2(\R^d)$, one has 
\begin{align}\label{conjecture}
\int_{\R} \chi(t) \left( \mathcal I _{sc}^\eps(t) \psi^\eps_0 - \mathcal U^\eps_H(t,t_0)\psi^\eps_0  \right) dt = 
 o(\sqrt\eps)
\end{align} 
 Estimates that are ``averaged in time'' have been previously obtained for systems (see~\cite{GMMP,FL08,FL17} for example). They correspond to an observation period that is a short, but non negligible, time interval.  
 
An approximation as \eqref{conjecture} can be easily proved for the toy-model~\eqref{ToMo} (see Section~\ref{sec:HKtoy}) below. The proof in the general case is work in progress~\cite{FLR2}. It involves more refined estimates than those of Theorem~\ref{thm:decoupling}, that we present in the next section.

The authors believe that the technics used for treating the apparition of a new contribution when trajectories reach a hypersurface (the crossing one in this special case) will be useful for developing Herman--Kluk approximations for avoided crossings and conical ones, using the hopping trajectories of~\cite{FL17} and~\cite{FL08} respectively. However, we point out, that conical crossings pose the additional difficulty that Gaussians states do not remain Gaussian, even at leading order, as emphasized in~\cite{Hag94}.

%%%%%%%%%%%%%%%%%%%%%%%%%%%%%%%%%%%%%%%%%%%%%%%%%%%%%

\section{A sketchy proof for Herman--Kluk approximations
% formula in the adiabatic case
}\label{sec:proofs}

We consider   here the scalar and the  adiabatic case and we discuss the proof of Theorems~\ref{thm:scalarhk} and~\ref{thm:decoupling}.

\subsection{The proof strategy}
As mentioned in the introduction, the underlying idea for constructing Gaussian based approximations 
for unitary propagators is to start form equation~\eqref{unitaryprop}. Let us develop a 
proof strategy based on this idea.
\begin{description}
\item[Step 1.]  For each $z\in\R^{2d}$, we build a thawed wave packet approximation $\psi^\eps_{{\rm th},z}(t)$ to the 
Schr\"odinger system \eqref{system} with initial data 
\[
\psi^\eps_{|t=t_0} =  \psi^\eps_0 = \left\{ \begin{array}{cc} g^\eps_z & \text{for}\ N=1,\\ 
\widehat{\vec V_0} g^\eps_z & \text{for}\ N>1.   \end{array}\right.
\]
We prove that $\psi^\eps_{{\rm th},z}(t)$ satisfies an evolution equation of the form
\begin{equation}\label{eq:thawed}
i\eps\partial_t \psi^\eps_{{\rm th},z}(t) = \widehat{H(t)}\psi^\eps_{{\rm th},z}(t) + \eps^{2} R^\eps_z(t),
\quad \psi^\eps_{{\rm th},z}(t_0) = \psi^\eps_0
\end{equation}
with source term $R^\eps_z(t)$. 
\item[Step 2.] For the thawed Gaussian propagation of general initial data 
\[
\psi^\eps_{|t=t_0} = \psi^\eps_0 =\left\{ \begin{array}{cc} v^\eps_0 & \text{for}\ N=1,\\ 
\widehat{\vec V_0} v^\eps_0 & \text{for}\ N>1,   \end{array}\right.
\]
with $v^\eps_0\in L^2(\R^d,\C)$ we consider
\[
{\mathcal I}^\eps_{\rm th}(t,t_0)\psi^\eps_0 = 
(2\pi\eps)^{-d}  \int_{z\in\R^{2d}} \langle  g^\eps_z,v^\eps_0\rangle 
\,\psi^\eps_{{\rm th},z}(t) \, dz.
\]
The error $e^\eps_{\rm th}(t) = {\mathcal U}^\eps_H(t,t_0)\psi^\eps_0 -  {\mathcal I}^\eps_{\rm th}(t,t_0)\psi^\eps_0$
satisfies the evolution equation
\[
i\eps\partial_t e^\eps_{\rm th}(t) = \widehat H(t) e^\eps_{\rm th}(t) + \eps^{2}\Sigma^\eps_{\rm th}(t),
\quad e^\eps_{\rm th}(t_0) = 0
\]
with source term
\[
\Sigma^\eps_{\rm th}(t) =  (2\pi\eps)^{-d} \int_{\R^{2d}} \langle g^\eps_z,\psi^\eps_0\rangle 
R^{\eps}_z(t) dz.
\]
Since $\widehat H(t)$ is self-adjoint, the usual energy argument provides 
\[
\|e^\eps_{\rm th}(t)\| \le \eps\int_{t_0}^t \|\Sigma^\eps_{\rm th}(s)\| ds.
\]
\item[Step 3.] For analysing the source term $\Sigma^\eps_{\rm th}(t)$, we consider the 
integral operator 
\[
\psi\mapsto (2\pi\eps)^{-d} \int_{\R^{2d}} \langle g^\eps_z,\psi\rangle 
R^{\eps}_z(t) dz
\]
and its Bargmann kernel 
\[
k_{\mathcal B}(t;X,Y) = (2\pi\eps)^{-2d} \int_{\R^{2d}} \langle g^\eps_z,g^\eps_Y\rangle 
\langle g^\eps_X, R^{\eps}_z(t)\rangle \,dz,\quad X,Y\in\R^{2d}.
\]
We aim at establishing constants $C_1(t),C_2(t)>0$, that do do not depend on the semiclassical 
parameter $\eps$, such that 
\begin{equation}\label{eq:schur}
\sup_{X} \int_{\R^{2d}} |k_{\mathcal B}(t;X,Y)| dY\le C_1(t),\ 
\sup_{Y} \int_{\R^{2d}} |k_{\mathcal B}(t;X,Y)| dX\le C_2(t),
\end{equation}
since then, by the Schur test,
\[
\|\Sigma^\eps_{\rm th}(t)\| \le \sqrt{C_1(t)C_2(t)}\ \|\psi^\eps_0\|.
\]
\item[Step 4.] For systems, that is, for $N>1$, we use the additional observation 
that  
\begin{align*}
&{\mathcal I}^\eps_{\rm th}(t,t_0)\psi^\eps_0 \\
&=(2\pi\eps)^{-d}  \int_{z\in\R^{2d}} \langle  g^\eps_z,v^\eps_0\rangle 
\vec V(t,t_0,\Phi^{t,t_0}_h(z)) {\rm e}^{\frac{i}{\eps}S(t,t_0,z)} g^{\Gamma(t,t_0,z),\eps}_{\Phi^{t,t_0}_h(z)}\, \, dz + O(\eps).
\end{align*}
\item[Step 5.] We turn the thawed propagation in a frozen one, in proving that
\[
\mathcal{I}^\eps_{\rm th}(t,t_0) = \mathcal{I}^\eps_{\rm fr}(t,t_0) + O(\eps)
\]
in the norm of bounded operators on $L^2(\R^d)$, where the frozen propagator 
is defined by the Herman--Kluk formula
\[
\mathcal{I}^\eps_{\rm fr}(t,t_0)\psi^\eps_0 = (2\pi\eps)^{-d}  \int_{z\in\R^{2d}} \langle  g^\eps_z,v^\eps_0\rangle 
\vec U(t,t_0,z)\, {\rm e}^{\frac{i}{\eps}S(t,t_0,z)} g^\eps_{\Phi^{t,t_0}_h(z)} \, dz
\]
with
\[
\vec U(t,t_0,z) = \left\{ \begin{array}{ll} u_0(t,t_0,z) & \text{for}\ N=1,\\
u_0(t,t_0,z) \vec V(t,t_0,\Phi^{t,t_0}_h(z)) & \text{for}\ N>1.\end{array}\right.
\]
\end{description}

Once the previous steps have been carried out, we have proven the basic $J=0$ version of the scalar 
Herman--Kluk formula of Theorem~\ref{thm:scalarhk} and its generalization 
to the adiabatic situation given in Theorem~\ref{thm:decoupling}.

\subsection{The thawed remainder}
We first consider scalar wave packet propagation as described in Theorem~\ref{thm:propag1} with an accuracy of 
 order $\eps$, that is, for $N_0=1$. The corresponding thawed Gaussian wave packet $\psi^\eps_{{\rm th},z}(t)$ 
that is defined by the right hand side of \eqref{eq:structure} satisfies an evolution equation of the form \eqref{eq:thawed} with a source term
\[
R^{\eps}_z(t) = {\rm e}^{\frac{i}{\eps}S(t,t_0,z)}\,{\rm op}_\eps^w(L_z(t,t_0)) 
g^{\Gamma(t,t_0,z),\eps}_{\Phi^{t,t_0}_h(z)},
\]
where $w\mapsto L_z(t,t_0,w)$ is a smooth function, that is polynomially bounded. It depends on the 
remainder of Taylor expansions of $h(t,\cdot)$ around the point $\Phi^{t,t_0}_h(z)$, 
see \cite[Section~4.3.1]{CR}. For adiabatic propagation by systems with eigenvalue gaps, as presented in 
Theorem~\ref{thm:decoupling}, we work with
\[
\psi^\eps_{{\rm th},z}(t) = {\rm e}^{\frac i\eps S(t,t_0,z)} \widehat {\vec V(t,t_0)} 
\left( 1 + \sqrt\eps\, \vec a(t,t,_0,z)\cdot\frac{x-q(t,t_0,z)}{\sqrt\eps}\right)g^{\Gamma(t,t_0,z),\eps}_{\Phi^{t,t_0}_h(z)},
\] 
where the vector $\vec a(t,t_0,z;x)$ can be constructed explicitly.  
This wave packet satisfies an evolution equation of the form \eqref{eq:thawed} with source term
\[
R^{\eps}_z(t) = {\rm e}^{\frac{i}{\eps}S(t,t_0,z)}\,{\rm op}_\eps^w(\vec L_z(t,t_0)) 
g^{\Gamma(t,t_0,z),\eps}_{\Phi^{t,t_0}_h(z)},
\]
where the vector-valued function $w\mapsto \vec L_z(t,t_0,w)$ contains remainder terms of 
Taylor expansions of $h(t,\cdot)$ around the classical trajectory. We note 
that $\vec L_z(t,t_0,\cdot)$ has contributions both in the range of $\vec V(t,t_0,\cdot)$ 
but also in the orthogonal complement.

\subsection{The Schur estimate}
We now analyse the Bargmann kernel of the source term $\Sigma^\eps_{\rm th}(t)$. Since
$$|\langle g^\eps_z,g^\eps_Y\rangle | =  {\rm e}^{-\frac{|Y-z|^2}{4\eps}},$$
we have 
\[
|k_{\mathcal B} (t;X,Y)| \le 
(2\pi\eps)^{-2d} \int_{\R^{2d}} {\rm e}^{-\frac{|Y-z|^2}{4\eps}} \, |\langle g^\eps_X,R^\eps_z(t)\rangle| \,\,dz.
\]
Hence, the crucial estimate that is required concerns the Bargmann transform of the remainder $R^\eps_z(t)$. 
We write the remainder as
\[
R^\eps_z(t) = {\rm e}^{\frac{i}{\eps} S(t,t_0,z)}\, {\rm op}_\eps^w(L_z(t,t_0)) g^{\Gamma_z,\eps}_{\Phi_z},
\]
where $w\mapsto L_z(t,t_0,w)$ is a smooth function on phase space, that grows at most polynomially, 
and $\Gamma_z = \Gamma(t,t_0,z)$, $\Phi_z = \Phi^{t,t_0}(z)$ are short-hand notations for 
the classical quantities defined in \eqref{def:flot1} and \eqref{def:Gamma}, respectively.
We express the Bargmann transform as a phase space integral
\[
\langle g^\eps_X,R^\eps_z(t)\rangle = {\rm e}^{\frac{i}{\eps} S(t,t_0,z)} \int_{\R^{2d}} L_z(t,t_0,w) 
\, {\rm Wig}(g_X^\eps,g_{\Phi_z}^{\Gamma_z,\eps})(w) \, dw 
\]
with respect to the cross-Wigner function of two Gaussian wave packets with different centers and different width. 
One can prove (see \cite[Lemma 5.20]{LL}) that
\begin{align*}
&{\rm Wig}(g_X^\eps,g_{\Phi_z}^{\Gamma_z,\eps})(w) \\
&=  (\pi\eps)^{-d}\, \gamma_{X,z}\,    
\exp\left(\frac{i}{\eps}J(X-\Phi_z)\cdot w + \frac{i}{2\eps} G_z(t)(w-m_{X,z})\cdot(w-m_{X,z})\right),
\end{align*}
where $m_{X,z} = \frac12(X+\Phi_z)$ is the mean of the two centres, while $\gamma_{X,z}$ is a complex number with $|\gamma_{X,z}|\le 1$ and $G_z(t)\in{\mathfrak S}^+(2d)$. Repeated integration by parts, see 
\cite[Proposition 5.21]{LL}, yields an upper bound 
\[
\mid\langle g^\eps_X, R^{\eps}_z(t)\rangle\mid \ \le\  c_z(t) \left\langle\frac{X-\Phi^{t,t_0}(z)}{\sqrt\eps}\right\rangle^{-(d+1)}, 
\]
where the constant $c_z(t)>0$ depends on bounds of the function $L_z(t,t_0,w)$ and is 
inversely proportional to the smallest eigenvalue of $\Im G_z(t)$. A combination of arguments 
given in the proofs of \cite[Lemma~5.18]{LL} and \cite[Lemma 3.2]{R}, reveals that 
the spectrum of $\Im G_z(t)$ is bounded away from zero uniformly in $z$, 
which implies the existence of constant $c(t)>0$ such that 
\[
\mid\langle g^\eps_X, R^{\eps}_z(t)\rangle\mid \ \le\  c(t) \left\langle\frac{X-\Phi^{t,t_0}(z)}{\sqrt\eps}\right\rangle^{-(d+1)}. 
\]
This gives us enough decay to deduce the existence of constants $C_1(t),C_2(t)>0$ such that the 
Schur estimate \eqref{eq:schur} holds.

 \subsection{Passing from thawed to frozen approximation}
 
Here we present a slight variant of \cite[Proposition~4.1]{R} for passing from 
a thawed to a frozen Gaussian approximation by a linear deformation argument. 

\begin{prop} Let $\vec U(t,t_0,z)$ be a smooth function with values in $\C^N$, $N\ge 1$, whose derivatives are 
at most of polynomial growth.
%\footnote{\textcolor{magenta}{can we control polynomial growth??}}
 Then,
\begin{align*}
&(2\pi\eps)^{-d} \int_{\R^{2d}} \langle g^\eps_z,\psi\rangle 
\,\vec U(t,t_0,z)\, {\rm e}^{\frac{i}{\eps} S(t,t_0,z)} g^{\Gamma(t,t_0,z),\eps}_{\Phi^{t,t_0}_h(z)} dz \\
&= 
(2\pi\eps)^{-d} \int_{\R^{2d}} \langle g^\eps_z,\psi\rangle 
u_0(t,t_0,z)\,\vec U(t,t_0,z)\, {\rm e}^{\frac{i}{\eps} S(t,t_0,z)} g^{\eps}_{\Phi^{t,t_0}_h(z)} dz + O(\eps)
\end{align*}
uniformly for all $\psi\in L^2(\R^d,\C)$ with norm one.
\end{prop}

\begin{proof}
For notational simplicity, we omit the time-dependance in $S = S(t,t_0,z)$, $\Phi = \Phi^{t,t_0}_h(z)$,
$\Gamma = \Gamma(t,t_0,z)$, and $\vec U = \vec U(t,t_0,z)$.   
We consider both operators, the thawed and the frozen one, as special members of a class of linear operators of the form 
\[
{\mathcal I}\psi = (2\pi\eps)^{-d} \int_{\R^{2d}} \langle g^\eps_z,\psi\rangle (x-\Phi_q(z))^\alpha\, 
\vec W(z)\, \e^{\frac{i}{\eps} S(z)}  
g_{\Phi(z)}^{{\mathcal G}(z),\eps} dz\,,
\]
that are defined by two smooth functions ${\mathcal G}:\R^{2d}\to\mathfrak{S}^+(d)$ and $\vec W:\R^{2d}\to\C^N$.
%\footnote{\textcolor{magenta}{Again: what growth conditions?}} 
The Siegel half-space $\mathfrak{S}^+(d)$ is invariant under inversion in the sense that any 
$G\in{\mathfrak S}^+(d)$ is invertible with $-G^{-1}\in{\mathfrak S}^+(d)$. We require that the smallest eigenvalue of $\Im({\mathcal G}(z))$ and $\Im(-\mathcal G^{-1}(z))$ are bounded away from zero 
uniformly in $z$. The monomial powers $(x-\Phi_q(z))^\alpha$ with $\alpha\in\N_0^d$ are included for technical reasons, that will become clear soon. These operators are bounded on $L^2(\R^d)$ and satisfy
%\footnote{\textcolor{magenta}{Shall we present the estimate like this or as a Lemma~\ref{lem:IPP}?}} 
\begin{equation}\label{lem:IPP}
\|{\mathcal I}\| \le C\, \eps^{\lceil|\alpha|/2\rceil},
\end{equation}
where the constant $C >0$ independent of $\eps$, and $\lceil|\alpha|/2\rceil$ denotes  
the smallest integer $\ge |\alpha|/2$, see \cite[Proposition~5.12]{LL}. 
We linearly link the thawed matrix function $z\mapsto\Gamma(z)$ and the frozen $z\mapsto i\Id$ by setting
\[
{\mathcal G}(z,s) = (1-s)\Gamma(z) + i s\Id\in\mathfrak{S}^+(d),\quad s\in[0,1],
\]
and consider the corresponding Gaussian function, that is only partially normalised, 
\[
\widetilde g^{{\mathcal G}(z,s),\eps}_{\Phi(z)}(x) = (\pi\eps)^{-d/4} 
\e^{\frac{i}{\eps}\Phi_p(z)\cdot(x-\Phi_q(z))+\frac{i}{2\eps}{\mathcal G}(z,s)(x-\Phi_q(z))\cdot(x-\Phi_q(z))}. 
\]
We now aim at constructing a smooth function $\vec W(z,s)$ with two properties. 
\begin{enumerate}
\item
Firstly, we require that
$\vec W(z,0) = {\rm det}^{-1/2}(A(z)+iB(z)) \vec U(z)$, 
ensuring that the deformation value $s=0$ corresponds to the thawed approximation. 
\item
Secondly, we hope to acchieve
\[
\frac{\partial}{\partial s}\, (2\pi\eps)^{-d} \int_{\R^{2d}} \langle g^\eps_z,\psi\rangle \vec W(z,s) \e^{\frac{i}{\eps} S(z)}  
\widetilde g_{\Phi(z)}^{{\mathcal G}(z,s),\eps} dz = O(\eps), 
\]
uniformly for all $\psi\in L^2(\R^d)$ of norm one. 
\end{enumerate}
Once this construction has been carried out, we will verify that  the deformation yields 
the frozen approximation for $s=1$, that is, $\vec W(z,1) = u_0(z)\vec U(z)$. 
%We start by differentiating the deformed Gaussian, 
%\[
%\frac{\partial}{\partial s}\widetilde g_{\Phi(z)}^{{\mathcal G}(z,s),\eps} = 
%\tfrac{i}{2\eps}\partial_s{\mathcal G}(z,s)(x-\Phi_q(z))\cdot(x-\Phi_q(z)) \, \widetilde g_{\Phi(z)}^{{\mathcal G}(z,s),\eps}
%\] 
%and observe that $\partial_s \vec W(z,s)$ has to absorb a quadratic polynomial in $x-\Phi_q(z)$. This becomes %possible 
As a first step, we open the inner product involving the standard Gaussian $g^\eps_z$ and examine the multi-variate exponential function
\[
\overline{g^\eps_z(y)} \e^{\frac{i}{\eps}S(z)} \widetilde g_{\Phi(z)}^{{\mathcal G}(z,s),\eps}.
\] 
Using the derivative properties of the action $S(z)$, one obtains the identity 
\begin{align*}
&(x-\Phi_q(z)) \overline{g^\eps_z(y)} \e^{\frac{i}{\eps}S(z)} \widetilde g_{\Phi(z)}^{{\mathcal G}(z,s),\eps} \\
&= 
\left(\frac{\eps}{i}\,M^{-T}_{{\mathcal G}(z,s)}(z)(i\partial_q+\partial_p)-f(x,z,s)\right) \overline{g^\eps_z(y)} \e^{\frac{i}{\eps}S(z)} \widetilde g_{\Phi(z)}^{{\mathcal G}(z,s),\eps}
\end{align*}
where
\[
M_{{\mathcal G}(z,s)}(z) = -i{\mathcal G}(z,s) A(z)-{\mathcal G}(z,s) B(z) + iC(z) + D(z)
\]
is an invertible complex $d\times d$ and 
\[
f(x,z,s) =  M^{-T}_{{\mathcal G}(z,s)}(z) \left(\tfrac12 \left((i\partial_{q_j}+\partial_{p_j}){\mathcal G}(z,s)\right)(x-\Phi_q(z))\cdot(x-\Phi_q(z))\right)_{j=1}^d
\]
is a vector-valued function, that is quadratic in $x-\Phi_q(z)$. Since $\frac{i}{2\eps}(x-\Phi_q(z))\cdot f(x,z,s)$ is cubic in $x-\Phi_q(z)$, we use \eqref{lem:IPP} and recognize its contribution as a term of order $\eps$.
Thus, we have
\begin{align*}
&(2\pi\eps)^{-d}\int_{\R^{2d}} \langle g^\eps_z,\psi\rangle \vec W(z,s) \e^{\frac{i}{\eps} S(z)}  
\partial_s \widetilde g_{\Phi(z)}^{{\mathcal G}(z,s),\eps} dz\\
&= 
(2\pi\eps)^{-d}\int_{\R^{2d}}  \vec W(z,s) \, L(z,s) (x-\Phi_q(z))\cdot (i\partial_q+\partial_p)
\langle g^\eps_z,\psi\rangle \e^{\frac{i}{\eps} S(z)} \widetilde g_{\Phi(z)}^{{\mathcal G}(z,s),\eps} dz + O(\eps),
\end{align*}
with
\[
L(z,s) = \tfrac{1}{2}\,M^{-1}_{{\mathcal G}(z,s)}(z)\partial_s{\mathcal G}(z,s). 
\] 
We perform an integration by parts and arrive at 
\begin{align*}
&-(2\pi\eps)^{-d}\int_{\R^{2d}} \sum_{k=1}^d (i\partial_{q_k}+\partial_{p_k})\left(\vec W(z,s)\, 
\left(L(z,s) (x-\Phi_q(z))\right)_k\right) \\
&\hspace*{12em}\times\langle g^\eps_z,\psi\rangle \e^{\frac{i}{\eps} S(z)} \widetilde g_{\Phi(z)}^{{\mathcal G}(z,s),\eps} dz + O(\eps).
\end{align*}
Computing the derivative we obtain several terms that are linear in $x-\Phi_q(z)$, and thus of order $\eps$. The contributions we have to keep are
\begin{align*}
&-\vec W(z,s)\sum_{k,\ell=1}^dL(z,s)_{k\ell}(i\partial_{q_k}+\partial_{p_k})(x-\Phi_q(z))_\ell \\
&= 
\vec W(z,s)\sum_{k,\ell=1}^dL(z,s)_{k\ell}(iA(z)+B(z))_{\ell k}
= \vec W(z,s)\,{\rm tr}\!\left(L(z,s)(iA(z)+B(z))\right).
\end{align*}
We observe that
\[
L(z,s)(iA(z)+B(z)) =-\tfrac{1}{2} \,M^{-1}_{{\mathcal G}(z,s)}(z) \partial_s M_{{\mathcal G}(z,s)}(z).
\] 
This suggests that $\vec W(z,s)$ should solve the differential equation
\[
\partial_s \vec W(z,s) - \tfrac{1}{2}{\rm tr}\!\left(M_{{\mathcal G}(z,s)}^{-1}(z)\partial_s M_{{\mathcal G}(z,s)}(z)\right) \vec W(z,s)= 0,
\] 
%By Liouville's formula, 
%\[
%\partial_s \det M_{{\mathcal G}(z,s)}(z) = \det M_{{\mathcal G}(z,s)}(z) \ {\rm tr}\!\left(M_{{\mathcal G}(z,s)}^{-1}(z)\,\partial_s M_{{\mathcal G}(z,s)}(z) \right),
%\]
%so that
that is, 
\[
\vec W(z,s) = 2^{-d/2}\ {\rm det}^{1/2}(M_{{\mathcal G}(z,s)}(z))\, \vec U(z),
\]
by using Liouville's formula for the differentiation of determinants. Checking for the initial condition at $s=0$, we observe that
\begin{align*}
M_{{\mathcal G}(z,0)}(z) &= -i\left(\Gamma(z)(A(z)-iB(z)) - (C(z)-iD(z))\right)\\
&= -i\left(\Gamma(z)-\overline{\Gamma(z)}\right) (A(z)-iB(z))\\
&= 2\ \Im\Gamma(z) \, (A(z)-iB(z)) = 2\ (A(z)+iB(z))^{-T}
\end{align*}
which implies $\vec W(z,0) = {\rm det}^{-1/2}(A(z)+iB(z))\vec U(z)$, indeed.
\end{proof} 
%
%\begin{lem}\label{lem:IPP}
%\footnote{\textcolor{magenta}{Shall we keep it?}}The type of integral operator ${\mathcal I}$ as defined above in \eqref{def:integralop} satisfies for all $\psi\in L^2(\R^d)$ 
%\[
%\|{\mathcal I}(\psi)\| \le C_\alpha\, \eps^{\lceil|\alpha|/2\rceil} \|\psi\|,
%\]
%where the constant $C_\alpha = C_\alpha(\Phi,S,\Gamma,u)>0$ is independent of $\eps$ and $\psi$, 
%but depends on the multi-index $\alpha$ as well as the functions $\Phi,S,\Gamma$ and $u$. 
%We note that $\lceil|\alpha|/2\rceil$ refers to the least integer $\ge |\alpha|/2$. 
%\end{lem}

%%%%%%%%%%%%%%%%%%%%%%%%%%%%

 \subsection{ Herman-Kluk approximation for the toy model}\label{sec:HKtoy}

The Schr\"odinger equation associated with~\eqref{ToMo} writes as a transport equation and, integrating along curves $s\mapsto (s,x+s)$ it reads 
 \begin{equation}\label{eq:toto}
 i\eps  \frac{d}{ds}\psi^\eps(s,x+s) = k(x+s)V_\theta(x+s)\psi^\eps(s,x+s)
 \end{equation}
with $\displaystyle{V_\theta(x)= \begin{pmatrix} 0&{\rm e}^{i\theta x}\\ {\rm e}^{-i\theta x}&0\end{pmatrix}}$.
 The
  equation reduces to 
 the  system of ODEs 
\beq\label{matad}
i\eps  \frac{d}{d\sigma} \eta^\eps(\sigma) =   {k\sigma} V_\theta(\sigma)\eta^\eps(\sigma).
\eeq
 This problem was solved  first in \cite{Fri56} then in a more general setting in  \cite{Hag89} where an asymptotic expansion in power of $\eps^{1/2}$ (with power of $\log\eps$ corrections),  at any order,  is established for  ${\mathcal R}^\eps_{k,\theta}(\sigma,\sigma_0)$, the propagator (or resolvent matrix)  of the linear differential equation \eqref{matad}; we shall use this result below .
 It is then possible to prove the Herman-Kluk approximation of the conjecture~\eqref{conjecture}.

  \begin{prop}\label{prop:HKTM}
 Consider $t_0<0$ and an  initial data of the form 
\begin{equation}\label{data:tm}
\psi^\eps_0= {\vec V_-} v^\eps_0 +O(\eps)\ \text{in}\  L^2(\R^d),
\end{equation}
where $\vec V_-$ is the smooth eigenvector for the minus mode (see~\eqref{eigenvalues:TM}). Set
\begin{align}\nonumber
\mathcal I^\eps_{sc}(v^\eps_0)= 
& (2\pi\eps)^{-d} \int_{z\in\R^{2d}} \langle v^\eps _0, g^\eps_z\rangle {\rm e}^{\frac i\eps S_-(t,t_0,q)} \vec V_-(t,t_0,q+t-t_0) g^\eps_{\Phi_-^{t,t_0}(z)}dz\\
\nonumber
&\, + \sqrt\eps \, \kappa_- 
(2\pi\eps)^{-d} \int_{z\in\R^{2d}} {\bf 1}_{q>0} {\bf 1}_{ t>t^\flat(z) } \langle v^\eps _0, g^\eps_z\rangle {\rm e}^{\frac i\eps (S_-(t_0,t^\flat(z),q)+S_+(t,t^\flat(z) ,0))}\\
\label{HKcros}
&\qquad \times\, 
 \vec V_+(t,t^\flat(z),t-t^\flat(z))
 g^{\eps, 1-ik}_{\Phi_+^{t,t^\flat (z)}(0,p^\flat(z))}
dz \end{align}
 where $\displaystyle{ \kappa
:= \sqrt{\frac {2\pi} {ik}} \frac{\theta}{2}} $ and  for $z=(q,p)$, $t^\flat(z)=t_0-q$, and $p^\flat(z)=p-kq$ have been computed in Example~\ref{gamma:TM}.
Then, 
 for all $\chi\in L^1(\R)$,  in $L^2(\R^d)$
$$
\int_{\R} \chi(t) \left( \mathcal I_{sc}^\eps (v^\eps_0)- 
\mathcal U^\eps_{H_{k,\theta}}(t,t_0)\psi^\eps_0(x) \right)dt =o(\sqrt\eps).
$$
 \end{prop}

The 
scalar Herman-Kluk prefactor is~$1$ because the width of the Gaussians wave packet stays constant along the propagation (see Ex.~\ref{clas:TM}). Note also (see Ex.~\ref{gamma:TM})
$$ \mathcal T^\flat g^{1}(y)= \e^{-\frac i{2\eps} ky^2} g^1= g^{1, (1-ik/2)};$$
in the statement above, we have chosen not to froze the Gaussian after crossing time. 
The value of the coefficient $\kappa$ arises from Theorem~\ref{theo:WPcodim1} and  Ex.~\ref{gamma:TM}.

\medskip

The proof  relies on the  analysis of the propagator ${\mathcal R}^\eps_{k,\theta}(\sigma,\sigma_0)$ as performed in~\cite{Hag89} (see p.~280 therein).
 We denote by 
 $Y_\pm(\sigma,\sigma_0) $
 the time-dependent  eigenprovectors of $V_\theta(\sigma)$ for the eigenvalues $E_\pm(\sigma)=\pm k\sigma$ (see Remark~\ref{pt:TM}):
 % These time dependent eigenvectors satisfy 
 \begin{equation}\label{prop:r}
 \vec Y_\pm(\sigma,\sigma_0)= \vec V_\pm(\sigma,\sigma_0,\sigma)=  \vec V_\pm(\sigma+r,\sigma_0+r,\sigma),\;\;\forall r\in\R.
 \end{equation}
 Then,
%  analysis of the solutions of the equation~\eqref{matad} as performed by~Friedrichs and as described by Hagedorn (see p.~280 in~\cite{Hag89}, of which we keep the notations) gives for 
  the solutions $\eta(\sigma)$ with initial data at time $\sigma_0<0$ of the form $\eta(\sigma_0)= \eta_0 \vec Y_- (\sigma_0)$:
\begin{enumerate}
\item If  $\sigma<0$, then 
$\displaystyle{\eta(\sigma)= \e^{\frac i \eps k \int_{\sigma_0}^\sigma \tau d\tau } \vec Y_-(\sigma,\sigma_0) \eta_0+O(\eps).}$
\item If $\sigma>0$, then
\begin{align*}
\eta(\sigma)= &\e^{\frac i \eps k \int_{\sigma_0}^\sigma \tau d\tau } \vec Y_-(\sigma,\sigma_0)\eta_0 \\
&\qquad - \sqrt\eps (1-i) \sqrt\pi\,  (-k) ^{-1/2} h(0) {\rm e}^{\frac i\eps \lambda}
 \e^{-\frac i \eps  k \int_{0}^\sigma \tau d\tau } \vec Y_+(\sigma,0) \eta_0
+O(\eps)
\end{align*}
with $h(0)=\left. \left( \vec Y_+(\sigma,0), \frac d{d\sigma} \vec Y_-(\sigma,0)\right)\right|_{\sigma=0} = -i\theta/2$ and $\lambda= k \int_{\sigma_0}^ 0\tau d\tau$. 
\end{enumerate} 
We observe that, with the notations of Ex.~\ref{gamma:TM},
$$ - (1-i) \sqrt\pi\,  (-k) ^{-1/2} h(0) = \sqrt{ \frac{2\pi}{ik}}\frac \theta 2 = \sqrt { \frac{2\pi}{ik}} \gamma^\flat=\kappa .$$

\begin{proof} [Proof of Proposition~\ref{prop:HKTM} ]
By~\eqref{eq:toto}, 
%\beq\label{RMsol}
 $$\psi^\eps(t,x): = \mathcal U^\eps_{H_{k,\theta}}(t,t_0)\psi^\eps_0(x) =  {\mathcal R}^\eps_{k,\theta}(x,x-t+t_0)\psi^\eps_0( x-t+t_0),$$
%\eeq
and, in view of Friedrichs' description, we   deduce 
\begin{align*}
\psi^\eps(t,x)=  &{\rm e}^{\frac {ik}{\eps}x (t-t_0) -\frac {ik}{2\eps}(t-t_0)^2 } \vec V_-(t,t_0,x) v^\eps_0(x-t+t_0) \\
&\qquad 
 + \sqrt\eps \kappa\ {\bf 1}_{0<x<t-t_0 } {\rm e}^{\frac {ik}{\eps}x(t-t_0)- \frac {ik}{2\eps} (t-t_0)^2  -\frac {i}{\eps} x^2}v^\eps_0(x-t+t_0)  \vec V_+(x,0,x) +o(\sqrt\eps) 
\end{align*}
where we have used 
$$\vec Y_-(x,x-t+t_0)= \vec V_-(t,t_0,x),\;\; \vec Y_+(\sigma,0)=\vec V _+(x,0,x)\;\;{\rm 
and}\;\;\lambda=-\frac k2 (x-t+t_0)^2.$$
Using~\eqref{reconstruction}, we write 
$$v^\eps_0(x-t+t_0) = (2\pi\eps)^{-d} \int_{z\in\R^{2d}} \langle v^\eps _0, g^\eps_z\rangle g^\eps_z(x-t+t_0) dz$$
and we observe that, in view of $S_-(t,t_0,z)= kq(t-t_0)+\frac k2 (t-t_0)^2$, we have 
$$g^\eps_z(x-t+t_0)={\rm e}^{- \frac {ik}{\eps}x (t-t_0)  +\frac{i}{2\eps}(t-t_0)^2} {\rm e}^{\frac i\eps S_-(t,t_0,z)} g^\eps_{\Phi_-^{t,t_0}(z)}(x)$$
Similarly, using
$$\displaylines{ 
S_+(t,t^\flat(z),q)=-\frac k 2 (t-t_0+q)^2,\;\;  S_-(t^\flat,t_0,q)= \frac{3k}2 q^2\cr
\mbox{and}\;\;\;\;
\Phi^{t,t^\flat(z)}_+(0,p^\flat(z))= (t-t_0+q, p-2kq-k(t-t_0)),
 ,\cr}$$
 we obtain 
%$$\displaylines{ 
%g^\eps_z(x-t+t_0) ={\rm e}^{\frac i\eps k(x-t+t_0)(2q+(t-t_0) }  g^\eps_{\Phi_+^{t,t^\flat (z)}(0,p^\flat(z))}(x)\hfill\cr\hfill 
%= {\rm e}^{\frac i\eps S_-(t^\flat,t_0,q)+\frac i\eps S_+(t,t^\flat,0) +  \frac{k} {i\eps} (x-q-t+t_0)^2}  g^\eps_{\Phi_+^{t,t^\flat (z)}(0,p^\flat(z))}(x)\\
%= {\rm e}^{\frac i\eps S_-(t^\flat,t_0,q)+\frac i\eps S_+(t,t^\flat,0) }  {\rm WP}^\eps_{\Phi_+^{t,t^\flat (z)}(0,p^\flat(z))}( \mathcal T^\flat g_0)
%\cr}$$
\begin{align*}
g^\eps_z(x-t+t_0) &={\rm e}^{\frac i\eps k(x-t+t_0)(2q+(t-t_0) }  g^\eps_{\Phi_+^{t,t^\flat (z)}(0,p^\flat(z))}(x)\\
&= {\rm e}^{\frac i\eps S_-(t^\flat,t_0,q)+\frac i\eps S_+(t,t^\flat,0) +  \frac{k} {i\eps} (x-q-t+t_0)^2}  g^\eps_{\Phi_+^{t,t^\flat (z)}(0,p^\flat(z))}(x)\\
&= {\rm e}^{\frac i\eps S_-(t^\flat,t_0,q)+\frac i\eps S_+(t,t^\flat,0) }  {\rm WP}^\eps_{\Phi_+^{t,t^\flat (z)}(0,p^\flat(z))}( \mathcal T^\flat g^{1})
\end{align*}
Putting these elements together, we are left with
\begin{align*}
&\psi^\eps(t,x)= o(\sqrt\eps) +  (2\pi\eps)^{-d} \int_{z\in\R^{2d}} \langle v^\eps _0, g^\eps_z\rangle {\rm e}^{\frac i\eps S_-(t,z) }
 \vec V_-(t,t_0 ,x )g^\eps_{\Phi_-^{t,t_0}(z)}dz\\
& \;\; + \sqrt\eps  \kappa 
(2\pi\eps)^{-d} \int_{z\in\R^{2d}} {\bf 1}_{0<x<t-t_0 } \langle v^\eps _0, g^\eps_z\rangle {\rm e}^{\frac i\eps S_-(t^\flat,t_0,q)+\frac i\eps S_+(t,t^\flat,0)}  \vec V_+(x, 0, x) g^\eps_{\Phi_+^{t,t^\flat }(z)}dz.
\end{align*}
Using Taylor expansion and Lemma~\ref{lem:IPP}, we can transform the first part of the right-hand side of the preceding equation: 
\begin{align*}
(2\pi\eps)^{-d}& \int_{z\in\R^{2d}} \langle v^\eps _0, g^\eps_z\rangle {\rm e}^{\frac i\eps S_-(t,z) }
 \vec V_-(t,t_0 ,x )g^\eps_{\Phi_-^{t,t_0}(z)}dz\\
 &=
 (2\pi\eps)^{-d} \int_{z\in\R^{2d}} \langle v^\eps _0, g^\eps_z\rangle {\rm e}^{\frac i\eps S_-(t,z) }
 \vec V_-(t,t_0 ,q_-(t) )g^\eps_{\Phi_-^{t,t_0}(z)}dz +O(\eps)
 \end{align*}
 in $L^2(\R^d)$, which allows to identify the first term of~\eqref{HKcros} since $q_-(t)=q+t-t_0$. In the second term, one can treat similarly the term $\vec V_+(x,0, x)$ that turns into (using also~\eqref{prop:r})
 $$\vec V(t-t_0+q,0,t-t_0+q)=\vec V(t,t_0-q,t-t_0+q)= \vec V(t,t^\flat(z), t-t^\flat(z)$$
It remains to turn the   discontinuous function $ x\mapsto {\bf 1}_{0<x<t-t_0 }$ into 
 $${\bf 1}_{0<q_+(t)<t-t_0 }= {\bf 1}_{q<0} {\bf 1}_{t>t_0-q}= {\bf 1}_{q<0} {\bf 1}_{t>t^\flat(z)}.$$
 One then takes advantage of the averaging in time to use the estimate of~\ref{lem:IPP} despite the discontinuity.
 %of this function of $x$. 
One  regularizes the discontinuous function which will differ from its regularization on a set of small Lebesgue measure in the variable~$t$.
 One can then use the preceding argument on the regularized term and gets rid of the correction ones  by  estimating them thanks to the estimate~\eqref{lem:IPP} and using the smallness of integrals on $\chi$ on sets of small Lebesgue measures. 
\end{proof} 
  %%%%%%%%%%%%%%%%%%%%%%%%%%%%%%%%%%%%%%%%%%%%%%%%%%%%%%%%%

\end{document}